\documentclass{article}

\begin{document}

\title{An elliptic boundary value problem for $G_{2}$-structures}
\author{Simon Donaldson}
\maketitle


\newtheorem{prop}{Proposition}
\newtheorem{lem}{Lemma}
\newtheorem{thm}{Theorem}
\newcommand{\bR}{{\bf R}}
\newcommand{\bC}{{\bf C}}
\newcommand{\hatrho}{{\hat{\rho}}}
\newcommand{\hatA}{\hat{A}}
\newcommand{\hatF}{\hat{F}}
\newcommand{\sthird}{{\textstyle \frac{1}{3}}}
\newcommand{\stthird}{{\textstyle \frac{2}{3}}}
\newcommand{\sfourthree}{{\textstyle \frac{4}{3}}}
\newcommand{\squart}{{\textstyle \frac{1}{4}}}
\newcommand{\sthreetwo}{{\textstyle \frac{3}{2}}}
\newcommand{\shalf}{{\textstyle \frac{1}{2}}}
\newcommand{\ualpha}{\underline{\alpha}}

\newenvironment{proof}[1][Proof]{\begin{trivlist}
\item[\hskip \labelsep {\bfseries #1}]}{\end{trivlist}}
\newenvironment{definition}[1][Definition]{\begin{trivlist}
\item[\hskip \labelsep {\bfseries #1}]}{\end{trivlist}}
\newenvironment{example}[1][Example]{\begin{trivlist}
\item[\hskip \labelsep {\bfseries #1}]}{\end{trivlist}}
\newenvironment{remark}[1][Remark]{\begin{trivlist}
\item[\hskip \labelsep {\bfseries #1}]}{\end{trivlist}}

\newcommand{\qed}{\nobreak \ifvmode \relax \else
      \ifdim\lastskip<1.5em \hskip-\lastskip
      \hskip1.5em plus0em minus0.5em \fi \nobreak
      \vrule height0.75em width0.5em depth0.25em\fi}

\ \ \ \ \ \ \ \ \ \ \ \  {\it Dedicated to Jean-Pierre Demailly, for his 60th birthday.}
\section{Introduction} 
The purpose of this paper is to develop a deformation theory for torsion-free $G_{2}$-structures on $7$-manifolds with boundary. This extends the well-established theory for closed manifolds, going back to Bryant and Harvey (see page 561 in \cite{kn:Br}) and further developed by Joyce \cite{kn:J}, \cite{kn:J1} and Hitchin \cite{kn:H1},\cite{kn:H2}. Recall that a torsion-free $G_{2}$-structure on an oriented $7$-manifold $M$ can be viewed as a closed  $3$-form $\phi$ which is \lq\lq positive''(in a sense we recall below) at each point of $M$  and which satisfies the nonlinear equation
\begin{equation}   d *_{\phi} \phi =0  , \end{equation} where $*_{\phi}$ is the $*$-operator of the Riemannian metric $g_{\phi}$ defined by $\phi$ (which we also recall below). Begin with the standard case when $M$ is a closed manifold and let $c$ be a class in $H^{3}(M;\bR)$. Write ${\cal P}_{c}$ for the set of positive $3$-forms representing $c$. Certainly a torsion-free structure $\phi$ defines a point in ${\cal P}_{c}$ with $c=[\phi]$. Conversely, if we have a $c$ such that ${\cal P}_{c}$ is not empty then, as Hitchin observed, solutions of equation (1) in ${\cal P}_{c}$ correspond to critical points of the volume functional
\begin{equation}   V(\phi)= {\rm Vol}(M,g_{\phi})   \end{equation}
on ${\cal P}_{c}$. In fact $d(*_{\phi}\phi)$ can be regarded as the derivative $dV$ of the volume functional on the infinite-dimensional space ${\cal P}_{c}$. The basic results of the standard theory can be summarised  as follows.
\begin{enumerate}\item The derivative $dV$ is a Fredholm section of the cotangent bundle of the quotient  ${\cal Q}_{c}={\cal P}_{c}/{\cal G}$ of ${\cal P}_{c}$ by the group ${\cal G}$ of diffeomorphisms of $M$ isotopic to the identity. Thus the kernel of the Hessian of the volume functional on ${\cal Q}_{c}$ at a solution of (1) is finite-dimensional. 
\item In fact this kernel is always $0$,  which implies that if $\phi$ is a solution of (1) and if $c'$ is sufficiently close to $c=[\phi]$ in $H^{3}(M)$ then there is a unique solution $\phi'$ in ${\cal Q}_{c'}$ close to $\phi$. (Throughout this paper, cohomology is always taken with real coefficients.) In other words the \lq\lq period map'' $\phi\mapsto [\phi]$ defines a local homeomorphism from the  moduli space of torsion-free $G_{2}$-structures to $H^{3}(M)$. 
\item In fact the Hessian of $V$ on ${\cal Q}_{c}$ is negative-definite. A solution of (1) gives a strict local maximum for the volume functional on ${\cal Q}_{c}$.  
\end{enumerate} 

Now we go on to the case of a compact, connected,  oriented manifold $M$ with non-empty boundary $\partial M$. If $\rho$ is a closed $3$-form on $\partial M$ we define an {\it enhancement} of $\rho$ to be an equivalence class of closed forms $\phi$ on $M$ which restrict to $\rho$ on the boundary, under the equivalence relation
$\phi\sim \phi+d\alpha$ for all 2-forms $\alpha$ which vanish on $\partial M$. So the set of enhancements is an affine space with tangent space $H^{3}(M,\partial M)$. There is an algebraic notion of a positive $3$-form on $\partial M$. One definition is that these are exactly the forms which extend to positive forms on some neighbourhood of $\partial M$ in $M$. Fix a closed positive form $\rho$ on $\partial M$ and enhancement $\hatrho$. We write ${\cal P}_{\hatrho}$ for the set of positive forms in the enhancement class (in general, ${\cal P}_{\hatrho}$ could be the empty set) and ${\cal Q}_{\hatrho}$ for the quotient by the identity component of the group of diffeomorphisms of $M$ fixing the boundary pointwise. The boundary value problem, which was introduced in \cite{kn:D} and which we consider further here, is to solve equation (1) for $\phi$ in ${\cal P}_{\hatrho}$. Just as before, this is the Euler-Lagrange equation for the Hitchin volume functional, which descends to a functional on ${\cal Q}_{\hatrho}$.

The author has only been able to extend the first of the three results from the standard theory above to this setting. That is (continuing the informal discussion---more precise technical statements are given later),  we will show below (Proposition 8) that the the derivative of the volume functional is a Fredholm section of the cotangent bundle of ${\cal Q}_{\hatrho}$. This comes down to showing that our problem can be set up as an elliptic boundary value problem. The crucial linear result is Theorem 1. The kernel of the Hessian at a critical point $\phi$ is a finite dimensional vector space $H_{\phi}$ but this is not $0$ in general. Similarly, we can show that the Hessian has finite index (i.e. a finite dimensional negative subspace) but we have not been able to show that the Hessian is semi-definite. We will discuss these questions at greater length in Section 5 below. In any event, we do know  cases in which the space $H_{\phi}$ is zero and in such cases we get a straightforward deformation theory for our problem: for any enhanced boundary data  sufficiently close to $\hatrho$ there is a unique solution to the corresponding boundary value problem close to $\phi$ (Theorem 2). In Section 5  we give one application to the existence of \lq\lq $G_{2}$-cobordisms'' between closed 3-forms on a Calabi-Yau 3-fold (Theorem 4). 

The authors's work is  supported by the Simons Collaboration Grant \lq\lq Special holonomy in Geometry, Analysis and Physics''.

\section{Review of standard theory}
We begin with some purely algebraic statements.
\begin{itemize}\item
A $3$-form $\phi\in\Lambda^{3}(V^{*})$ on an oriented $7$-dimensional real vector space $V$ is called positive if the $\Lambda^{7}V^{*}$-valued quadratic form on $V$
\begin{equation} v\mapsto i_{v}(\phi)\wedge i_{v}(\phi)\wedge \phi \end{equation}
is positive definite. We fix a Euclidean structure $g_{\phi}$ in this conformal class by normalising so that $\vert \phi\vert^{2}=7$. Then, as in the Introduction, we have a $4$-form $*_{\phi}\phi$ which we also write as $\Theta(\phi)$. So $\Theta$ is a smooth map from the space of positive 3-forms on $V$ to $\Lambda^{4}V^{*}$. The  positive $3$-forms on $V$ form a single orbit under the action of $GL^{+}(V)$, so they are all equivalent. A convenient standard model for this paper is to take $V= \bR\oplus \bC^{3}= \{ (t, z_{1}, z_{2}, z_{3})\}$ and 
\begin{equation}   \phi= \omega \wedge dt + {\rm Im} (dz_{1}dz_{2}dz_{3}), \end{equation} where $\omega$ is the standard symplectic form $\sum dx_{a}\wedge dy_{a}$ on $\bC^{3}$. 
\item The stabiliser in $GL(V)$ of a positive $3$-form is isomorphic to the exceptional Lie group $G_{2}$. Under the action of this group the forms decompose as
\begin{equation}  \Lambda^{2}= \Lambda^{2}_{7}\oplus \Lambda^{2}_{14}\ \ \ \ \Lambda^{3}= \Lambda^{3}_{1}\oplus \Lambda^{3}_{7}\oplus \Lambda^{3}_{27}. \end{equation}
Here $\Lambda^{2}_{7}$ is the image of $V$ under the map $v\mapsto i_{v}(\phi)$ and $\Lambda^{2}_{14}$ is the orthogonal complement; $\Lambda^{3}_{7}$ is the image of $V$ under the map $v\mapsto i_{v}(*_{\phi} \phi)$,\  $\Lambda^{3}_{1}$ is the span of $\phi$ and $\Lambda^{3}_{27}$ is the orthogonal complement of their sum. 

We have a quadratic form on $\Lambda^{2}$ defined by $\alpha\mapsto \alpha\wedge\alpha\wedge \phi$. The eigenspaces of this form, relative to the standard Euclidean structure, are $\Lambda^{2}_{7}, \Lambda^{2}_{14}$. For $\alpha_{7}\in \Lambda^{2}_{7}$
\begin{equation}  \alpha_{7}\wedge \alpha_{7} \wedge \phi = 2 \vert \alpha_{7}\vert^{2} {\rm vol}, \end{equation}
and for $\alpha_{14}\in \Lambda^{2}_{14}$
\begin{equation}   \alpha_{14}\wedge \alpha_{14} \wedge \phi =- \vert \alpha_{14}\vert^{2} {\rm vol}. \end{equation}
\item The volume form {\rm vol} is a $\Lambda^{7}$-valued function on the open set of positive $3$-forms. Its derivative is given by
\begin{equation}   {\rm vol}(\phi+\delta \phi) = {\rm vol}(\phi) + \sthird \delta\phi \wedge \Theta(\phi)+ O(\delta\phi^{2}). \end{equation}
To identify the second derivative we write 
$$  \delta\phi= \delta_{1} \phi + \delta_{7}\phi + \delta_{27} \phi, $$
according to the decomposition (5). Then
$$   {\rm vol}(\phi+\delta \phi) = {\rm vol}(\phi) + \sthird\delta\phi
\wedge (*_{\phi}\phi)+ \stthird q(\delta\phi) {\rm vol}(\phi) + O(\delta \phi^{3}), $$ where $q$ is the quadratic form
\begin{equation} q(\delta\phi)= \sfourthree \vert \delta_{1}\phi\vert^{2} + \vert \delta_{7}\phi\vert^{2}- \vert \delta_{27}\vert^{2}. \end{equation}

This formula also gives the derivative of the map $\Theta$ (\cite{kn:J1}, Prop. 10.3.5):
\begin{equation} \Theta(\phi+\delta \phi)= \Theta(\phi)+\left( \sfourthree *_{\phi} \delta_{1}\phi + *_{\phi} \delta_{7}\phi- *_{\phi}\delta_{27}\phi\right) + O(\delta\phi^{2}). \end{equation} 

\end{itemize}

Now let $M$ be an  oriented $7$-manifold and $\phi$ be a positive $3$-form on $M$ which defines a torsion-free $G_{2}$-structure, so both $\phi$ and $*_{\phi}\phi$ are closed forms.  We can   decompose the exterior derivative according to the decomposition of the forms
$$   \Omega^{2}= \Omega^{2}_{7}\oplus  \Omega^{2}_{14}\ \ \ \Omega^{3}= \Omega^{3}_{1}\oplus\Omega^{3}_{7}\oplus\Omega^{3}_{27}. $$
The resulting operators satisfy various identities, akin to the K\"ahler identities on K\"ahler manifolds. The following Proposition states the main identities we will need in this paper (there is a comprehensive treatment in \cite{kn:Br2}).   Write $\chi:\Lambda^{1}\rightarrow \Lambda^{3}_{7}$ for the bundle isomorphism $\chi(\eta) =*_{\phi} (\eta\wedge \phi)$. We also usually write $*$ for $*_{\phi}$ and $d^{*}$ for the usual adjoint constructed using the metric $g_{\phi}$.
\begin{prop}
\begin{enumerate} \item The component $d_{1}:\Omega^{2}_{14}\rightarrow \Omega^{3}_{1}$ is identically zero.
\item The component
$d_{7}:\Omega^{2}_{14}\rightarrow \Omega^{3}_{7}$
is equal to the composite  $\squart \chi\circ d^{*}$ where $d^{*}:\Omega^{2}_{14}\rightarrow \Omega^{1}, $ and the component
$  d_{7}:\Omega^{2}_{7}\rightarrow \Omega^{3}_{7} $
is equal to the composite  $-\shalf\chi\circ d^{*}$ where
$   d^{*}:\Omega^{2}_{7}\rightarrow \Omega^{1}. $
\item For $d_{7}:\Omega^{1}\rightarrow \Omega^{2}_{7}$ and $d_{14}:\Omega^{1}\rightarrow \Omega^{2}_{14}$ we have 
$$ d^{*}d_{14}= 2d^{*}d_{7}=\stthird d^{*}d$$ on $\Omega^{1}$. 
\end{enumerate}
\end{prop}

\

\begin{proof}
For the first item, it suffices to prove that for  a compactly supported $\alpha\in \Omega^{2}_{14}$ and function $f$   the $L^{2}$-inner product $\langle d \alpha, f \phi\rangle$ is zero. This inner product is
$$ \int_{M} d\alpha\wedge f *\phi= - \int_{M} \alpha \wedge df \wedge *\phi, $$
(using $d *\phi=0$) which vanishes since $df\wedge* \phi$ lies in $\Omega^{5}_{7}$.  For the second item we consider first an $\alpha\in \Omega^{2}_{14}$ as above and the inner product $\langle d^{*}\alpha, \eta\rangle$ for a $1$-form $\eta$. By definition this is $\langle \alpha, d\eta\rangle$ and by (7) the latter can be expressed as
$$           - \int_{M}\alpha\wedge d\eta\wedge \phi . $$
By Stokes' Theorem (using, this time, $d\phi=0$) this is
$$    \int_{M} d\alpha\wedge\eta\wedge\phi=  \langle d\alpha, *(\eta\wedge\phi)\rangle=  \langle d_{7}\alpha, \chi(\eta)\rangle .$$
 One computes readily that for any $\eta$ we have
 $$ \vert \chi(\eta)\vert^{2}= 4 \vert \eta\vert^{2}, $$
and it follows that $d_{7}\alpha= \squart \chi\circ d^{*}\alpha$. The argument for the second part of the second item---for $\alpha\in\Omega^{2}_{7}$---is the same using (6). 

For the  third item: the equality $d^{*}d_{14}=2 d^{*}d_{7}$  follows from the second item and the fact that the component of $d^{2}$ from $\Omega^{1}$ to $\Omega^{3}_{7}$ is zero. The equality $d^{*}d_{14}= \stthird d^{*}d$ follows in turn because $d^{*}d= d^{*} d_{7}+d^{*}d_{14}$. \qed

\end{proof}

The variation of the volume functional (2) with respect to compactly supported variations of $\phi$ makes sense, even if $M$ is not compact. Suppose for the moment that $\phi$ is any closed positive $3$-form on $M$ and that $\alpha$
is a $2$-form with compact support. The pointwise formula (8) and integration by parts give
$$ {\rm Vol}(\phi+d\alpha) = {\rm Vol}(\phi) -\sthird \int_{M} \alpha\wedge d\Theta(\phi) + O(\alpha^{2}), $$
which shows that the torsion-free condition $d\Theta(\phi)=0$ is the Euler-Lagrange equation associated to the volume functional for exact variations. For any such $\phi$ we have
\begin{equation}  *_{\phi} d\Theta(\phi)\in \Omega^{2}_{14}(M). \end{equation}
This follows by direct calculation or, more conceptually, from the diffeomorphism invariance of the volume functional (see \cite{kn:D0}, Lemma 1). 

Now go back to assuming that $\phi$  defines a torsion-free $G_{2}$-structure, {\it i.e.} $d\Theta(\phi)=0$. For $\alpha $ in $\Omega^{2}(M)$ we define
\begin{equation}  W(\alpha)= *_{\phi} d \Theta(\phi+d\alpha) \end{equation}
so the equation $W(\alpha)=0$ is the torsion-free equation, for such variations. For any $\alpha$ we have $d^{*}W(\alpha)=0$ and $W(\alpha)$ takes values in the sub-bundle $$*_{\phi}\Lambda^{5}_{14,\phi+d\alpha}\subset \Lambda^{2}, $$ in an obvious notation. Let $L$ be the linearisation of the nonlinear operator $W$ at $\alpha=0$, {\it i.e.} 
$  W(\alpha) = L(\alpha) + O(\alpha^{2})$. 
By (10) this linearised operator is given by the formula
\begin{equation}
   L(\alpha) = \sfourthree d^{*}d_{1} \alpha + d^{*} d_{7}\alpha - d^{*}d_{27}\alpha.  \end{equation}

\begin{prop}
The linear operator $L$ vanishes on $\Omega^{2}_{7}$ and takes values in $\Omega^{2}_{14}$. For $\alpha=\alpha_{7}+\alpha_{14}$ we have
$$  L(\alpha)= d^{*}d_{7}\alpha_{14}- d^{*}d_{27}\alpha_{14} =-\Delta\alpha_{14}+ \sthreetwo d_{14}d^{*}\alpha_{14} $$
\end{prop}
\begin{proof}
The fact that $L$ vanishes on $\Omega^{2}_{7}$ follows from diffeomorphism invariance (or by direct calculation). Similarly, the fact that $L$ takes values in $\Omega^{2}_{14}$ is a consequence of the fact above that $W(\alpha)$ is a section  of $*_{\phi}\Lambda^{5}_{7,\phi+d\alpha}$ (or can be shown by direct calculation). The formulae for $L(\alpha)$ follow from items (1) and (3) in Proposition 1. \qed\end{proof}

There is a similar discussion for the Hessian of the volume functional. For $\alpha$ of compact support
$  {\rm Vol}(\phi+d\alpha)= {\rm Vol}(\phi) + \stthird Q(\alpha)$ where
\begin{equation}  Q(\alpha) = \sfourthree \Vert d_{1}\alpha \Vert^{2} + \Vert d_{7}\alpha\Vert^{2} - \Vert d_{27}\alpha\Vert^{2}. \end{equation}

This can also be expressed,  for $\alpha=\alpha_{7}+\alpha_{14}$, as
\begin{equation} Q(\alpha) = \Vert d_{7}\alpha_{14}\Vert^{2}- \Vert d_{27}\alpha_{14}\Vert^{2}= \langle L\alpha_{14},\alpha_{14}\rangle. \end{equation}
But we should emphasise that (14) is the \lq\lq primary'' formula (derived pointwise on $M$) and the passage to the expressions in (15) involves an application of Stokes' Theorem.

\ 

\

 We will now sketch a treatment of the standard results for closed manifolds mentioned in the Introduction. In this sketch we will just give a formal treatment, ignoring the analytical aspects, but these will be taken up in a more general setting in Section 4. To avoid unimportant complications we suppose here that $H^{2}(M)=0$.

   With $c=[\phi]\in H^{3}(M)$, the tangent space of ${\cal P}_{c}$ at $\phi$ is 
$$   T{\cal P}_{c}= {\rm Im} ( d:\Omega^{2}\rightarrow \Omega^{3}).
$$
The infinitesimal action of the group ${\cal G}$ of diffeomorphisms of $M$ is by the Lie derivative. But, since $\phi$ is closed,  for a vector field $v$ we have
$$  {\cal L}_{v}\phi= d(i_{v}(\phi))$$
and the 2-forms $i_{v}(\phi)$ are exactly $\Omega^{2}_{7}$. So the tangent space of ${\cal Q}_{c}$ at $\phi$ is 
$$  T{\cal Q}_{c}= \frac{  {\rm Im} ( d:\Omega^{2}(M)\rightarrow \Omega^{3}(M))}{ d\Omega^{2}_{7}}. $$
Let $\pi: \Omega^{2}_{14}\rightarrow  T{\cal Q}_{c}$ be the map induced by exterior derivative. This is obviously surjective and the kernel consists of those $\alpha_{14}\in \Omega^{2}_{14}$ such that there is an $\alpha_{7}\in \Omega^{2}_{7}$ with $d\alpha_{7}=d\alpha_{14}$. Under our assumption that $H^{2}(M)=0$ this means that $\alpha_{14}-\alpha_{7}=d\eta$ for some $\eta\in \Omega^{1}$, so $\alpha_{14}=d_{14}\eta$. Conversely, if $\alpha_{14}=d_{14}\eta$ we can define $\alpha_{7}=-d_{7}\eta$. So we see that the kernel of $\pi$ is the image of $d_{14}:\Omega^{1}\rightarrow \Omega^{2}_{14}$ and
$$    T{\cal Q}_{c}= \Omega^{2}_{14}/ {\rm Im}\ d_{14}. $$
By standard elliptic theory this can be identified with the kernel of the adjoint:
$$    T{\cal Q}_{c}= {\rm ker}\ d^{*}:\Omega^{2}_{14}\rightarrow \Omega^{1}. $$

Now by items (1) and (2) of Proposition 1,  for $\alpha\in {\rm ker}\  d^{*}\subset \Omega^{2}_{14}$ the only component of $d\alpha$ is $d_{27}\alpha\in \Omega^{3}_{27}$. It follows that for such $\alpha$ the linearised operator $L(\alpha)$ is $-\Delta \alpha$. In other words, after taking account of the diffeomorphism group action in this way, the linearised operator is
$$  -\Delta: \left( {\rm ker}\  d^{*}\subset\Omega^{2}_{14}\right)\rightarrow \left( {\rm ker}\ d^{*}\subset \Omega^{2}_{14}\right), $$ which is invertible. Similarly, with this representation of $T{\cal Q}_{c}$ the Hessian is
$$  Q(\alpha)= -\Vert d\alpha\Vert^{2},$$
which is negative definite.

\section{The boundary value problem--linear theory}

This section represents the heart of this paper, in which we set up a linear elliptic boundary value problem. We suppose, as in the Introduction, that $M$ is a compact, connected,  oriented 7-manifold with non-empty boundary and that the $3$-form $\phi$ defines a torsion-free $G_{2}$-structure on $M$. We will give a representation of the tangent space of ${\cal Q}_{\hatrho}$ at $\phi$ for the enhanced boundary value $\hatrho$ determined by $\phi$ and study the linearisation of the torsion-free equation in this representation.

As a preliminary,  note that there are two different notions of \lq\lq restriction to the boundary'' of a $p$-form $\sigma$ on $M$. One is that the pull back under the inclusion map, an element of $\Omega^{p}(\partial M)$, vanishes. We will denote this by the usual notation $\sigma\vert_{\partial M}=0$. For the other, stronger, notion we mean that the restriction vanishes regarded as a section of the  bundle $\Lambda^{p}T^{*}M\vert_{\partial M}$. We will denote this by the notation $\sigma\Vert_{\partial M}=0$.

To begin we define a vector space
\begin{equation}  T{\cal Q}= \frac{ \{d\gamma: \gamma \in \Omega^{2}(M), \gamma\vert_{\partial M}=0\} }{ \{ d\beta: \beta\in \Omega^{2}_{7}(M), \beta\Vert_{\partial M}=0\}}. \end{equation}
The definition of our space ${\cal P}_{\hatrho}$ and the identification of the vector fields on $M$ with $\Lambda^{2}_{7}$ suggests that this should represent the tangent space of the infinite dimensional manifold ${\cal Q}_{\hatrho}$ but we postpone any precise treatment of this for the present and just take (16) as a definition. Similarly we do not at this stage consider any topology on the vector space. In the same vein, we  define a vector space $H_{\phi}\subset T{\cal Q}$ to be
\begin{equation}  H_{\phi}= \frac{ \{d\gamma: \gamma \in \Omega^{2}(M), L(\gamma)=0,
\gamma\vert_{\partial M}=0\} }{ \{ d\beta: \beta\in \Omega^{2}_{7}(M), \beta\Vert_{\partial
M}=0\}}. \end{equation}
This is the space of solutions of the linearised equations modulo infinitesimal diffeomorphisms. Note that on the right hand side of  (17) the denominator is a subspace of the numerator since $L$ vanishes on $\Omega^{2}_{7}$.

We now discuss the linear algebra of the decomposition $\Lambda^{2}_{7}\oplus \Lambda^{2}_{14}$ on the  boundary. We have a $3$-form $\rho=\phi\vert_{\partial M}$ and a $2$-form $\omega\in \Omega^{2}(\partial M)$ given by $\omega=i_{\nu}(\phi)$ where $\nu$ is a unit outward-pointing normal. At a point $p$ on $\partial M$ the situation corresponds to the model (4) on $\bR\oplus \bC^{3}$. There is a complex structure on the tangent space of $\partial M$ at $p$; the $2$-form $\omega$ is a positive $(1,1)$ form and $\rho$ is the real part of a complex volume form. In terms of the splitting $TM= T\partial M \oplus \bR \nu$ at $p$ the form $\phi$ is
$$  \phi= \omega \wedge \nu^{*} + \rho ,$$
where $\nu^{*}$ is the $1$-form dual to $\nu$.
We define a bundle map $$  \chi_{6}: \Lambda^{1}(\partial M)\rightarrow \Lambda^{2}(\partial M)$$
by $$  \chi_{6}( i_{v}\omega) = i_{v}(\rho), $$ for $v\in T\partial M$. 
We have a decomposition 
\begin{equation}  \Lambda^{2}(\partial M)= \Lambda^{2,\partial}_{6}\oplus \Lambda^{2,\partial}_{8}\oplus \Lambda^{2,\partial}_{1} \end{equation}
where, in terms of the complex structure, the summand $\Lambda^{2,\partial}_{6}$ consists of the real parts of forms of type $(2,0)$, the summand $\Lambda^{2,\partial}_{8}$ consists of the real $(1,1)$ forms orthogonal to $\omega$ and $\Lambda^{2,\partial}_{1}$ is the $1$-dimensional space spanned by $\omega$. Then $\chi_{6}$ is a bundle isomorphism from $\Lambda^{1}(\partial M)$ to $\Lambda^{2,\partial}_{6}$. 

\begin{lem}
At a boundary point:
\begin{enumerate}
\item
$$\Lambda^{2}_{7}= \Lambda^{2,\partial}_{1} \oplus \{ a\wedge \nu^{*} + \chi_{6}(a): a\in \Lambda^{1}(\partial M)\}, $$
\item $$\Lambda^{2}_{14}= \Lambda^{2,\partial}_{8} \oplus \{ 2 a \wedge \nu^{*} - \chi_{6}(a): a \in \Lambda^{1}(\partial M)\}. $$
\end{enumerate}
\end{lem}

This is straightforward to check, from the definitions. For a form $\alpha\in \Omega^{2}_{14}(M)$ we write $\alpha\Vert_{\partial, 8}$ for the section of $\Lambda^{2,\partial}_{8}$ over $\partial M$ defined by the decomposition in the second item of Lemma 1. Note that for a 2-form $\alpha$ in either of the spaces $\Omega^{2}_{7}, \Omega^{2}_{14}$ the two notions $\alpha\vert_{\partial M}=0, \alpha\Vert_{\partial M}=0$ are equivalent. Now define a vector space
\begin{equation}   \hatA= \{\alpha\in \Omega^{2}_{14}(M): d^{*} \alpha= 0, \alpha\Vert_{\partial, 8}=0\}. \end{equation}

We define a linear map $F_{\hatA}: \hatA\rightarrow T{\cal Q}$ as follows. It is clear from Lemma 1 that if $\alpha_{14}\in\Omega^{2}_{14}$ satisfies $\alpha_{14}\Vert_{\partial, 8}=0$ we can find a form $\beta_{7}\in \Omega^{2}_{7}$ such that $(\alpha_{14}+ \beta_{7})\vert_{\partial M}=0$. For $\alpha_{14}\in \hatA$ we define $F_{\hatA}(\alpha_{14})$ to be the equivalence class of $d(\alpha_{14} +\beta_{7})$ in $T{\cal Q}$. The definition of $T{\cal Q}$ means that this is well-defined, independent of the choice of $\beta_{7}$.

We digress here to review some standard Hodge Theory for manifolds with boundary. For any $p$ we consider the Laplace operator $\Delta:\Omega^{p}(M)\rightarrow \Omega^{p}(M)$ and the equation  with boundary conditions 
\begin{equation}   \Delta\mu=\rho\ , \ \mu\vert_{\partial M}=0\ ,\  d^{*}\mu\vert_{\partial M}=0. \end{equation}

\begin{prop}
\begin{itemize}\item If $d^{*}\rho=0$ then any solution $\mu$ of (20) satisfies $d^{*}\mu=0$. In fact this holds without assuming that $\mu\vert_{\partial M}=0$. 
\item If $\rho=0$ then a solution satisfies $d\mu=0$. The space of such solutions
$$  {\cal H}^{p}= \{\mu \in \Omega^{p}: d\mu=0,d^{*}\mu=0, \mu\vert_{\partial M}=0\}, $$
represents the relative cohomology group $H^{p}(M,\partial M)$.
\item There is a solution of (20) if and only if $\rho$ is $L^{2}$-orthogonal to ${\cal H}^{p}$. In this case we define $G\rho$ to be the unique solution $\mu$ orthogonal to ${\cal H}^{p}$.
\item If $\rho=d^{*}\sigma$ for some $\sigma\in \Omega^{p+1}$ then $\rho$ is orthogonal to ${\cal H}^{p}$.

\end{itemize}\end{prop}

Apart from the third item, the proofs are  straightforward variants of the usual theory for closed manifolds, checking boundary terms. The third item is an application of  elliptic boundary value theory. These results go back to Spencer and Duff, Morrey and Friedrichs. A  standard modern reference is  Section 2.4 in \cite{kn:S}. 

  With this theory at hand we can return to the space $\hatA$. 

\begin{lem}

For $\alpha\in {\cal H}^{2}$ the component $\pi_{14} (\alpha)\in \Omega^{2}_{14}$ lies in $\hatA $ and the induced  map $\pi_{14}:{\cal H}^{2}\rightarrow \hatA$ is injective. 
\end{lem}

\begin{proof}
First suppose that $\alpha=\alpha_{7}+\alpha_{14}$ lies in ${\cal H}^{2}$. Since $d^{*}\alpha$ and $d_{7}\alpha$ both vanish it follows from the second item  of Proposition 1 that $d^{*}\alpha_{14}=0$. Since $\alpha$ vanishes on the boundary it follows that $\alpha_{14}\Vert_{\partial,8}=0$ and thus $\alpha_{14}\in \hatA$. Suppose that $\alpha_{14}=0$. Then $\alpha_{7}\vert_{\partial M}=0$ and,  as we noted above,  this  implies that $\alpha_{7}\Vert_{\partial M}=0$. The 2-form $\alpha_{7}$ is harmonic and by the general theory (see \cite{kn:J1} Section 3.5.2, for example) the Bochner formula on $\Omega^{2}_{7}$ is the same as that on $\Omega^{1}$. Thus we have $\nabla^{*}\nabla \alpha_{7}=0$ (since the Bochner formula on $\Omega^{1}$ involves the Ricci curvature, which vanishes in our case). Now integration by parts, using the boundary condition $\alpha_{7}\Vert_{\partial M}=0$, shows that $\nabla\alpha_{7}=0$, and since $\alpha_{7}$ vanishes on the boundary it must be zero  everywhere.  This shows that $\pi_{14}$ induces an injection from ${\cal H}^{2}$ to $\hatA$. \qed \end{proof}

Define $A\subset \hatA$ to be the orthogonal complement of $\pi_{14}{\cal H}^{2}$.

\begin{prop}
The map $F_{\hatA}:\hatA\rightarrow T{\cal Q}$ is surjective with
kernel $\pi_{14}({\cal H}^{2})$. Hence there is an induced isomorphism
$F_{A}:A\rightarrow T{\cal Q}$.\end{prop}
\begin{proof}
Consider any $\gamma=\gamma_{14}+\gamma_{7}\in\Omega^{2}(M)$ with $\gamma\vert_{\partial M}=0$.  We apply Proposition 2 with $\rho= d^{*}\gamma_{14}$, so we find an $\eta=G(d^{*}\gamma_{14})$ with $\eta\vert_{\partial M}=0$ and $d^{*}d\eta=d^{*}\gamma_{14}$.
Now, by the third item of Proposition 1, we have $d^{*} d_{14}\eta= \stthird d^{*}d\eta=\stthird d^{*}\gamma_{14}$. This means that $\alpha_{14}= \gamma_{14} - \sthreetwo d_{14}\eta$ satisfies $d^{*}\alpha_{14}=0$. Also, since $\eta$ vanishes on the boundary so does $d\eta$, and this means that $\alpha_{14}\Vert_{\partial, 8}=0$. Thus $\alpha_{14}$ lies in $\hatA$. Going back to the definition of $F_{\hatA}$: the form $$\Gamma=\alpha_{14}+ \gamma_{7}-\sthreetwo d_{7}\eta= \gamma-\sthreetwo d\eta$$ vanishes on the boundary,  so  $F_{\hatA}(\alpha_{14})$ is the equivalence class of $d\Gamma$ in $T{\cal Q}$. But $d^{2}\eta=0$ so $d\Gamma=d\gamma$. This shows that $F_{\hatA}$ is surjective. 

In the other direction, suppose that $F_{\hatA}(\alpha_{14})=0$, for some $\alpha_{14}\in \hatA$. This means that we can choose an $\alpha_{7}\in \Omega^{2}_{7}$ such that $\alpha=\alpha_{7}+\alpha_{14}$ restricts to zero on the boundary and $d\alpha=0$. As in the proof of Lemma 2, the condition $d^{*}\alpha_{14}=0$ implies that $d^{*}\alpha=0$, so $\alpha$ lies in ${\cal H}^{2}$ and $\alpha_{14}$ is in $\pi_{14}({\cal H}^{2})$.\qed \end{proof}

\

We can now set-up the linear boundary value problem which is the main point of this paper. 
\begin{thm}
  For $\rho$ in $\Omega^{2}_{14}(M)$ the equation $\Delta \alpha=\rho$ for $\alpha\in \Omega^{2}_{14}$, with boundary conditions
$$   \alpha\Vert_{\partial, 8}=0\ , \   d^{*}\alpha\vert_{\partial M}=0, $$
is a self-adjoint elliptic boundary value problem. Moreover if $d^{*}\rho=0$ then a solution $\alpha$ satisfies $d^{*}\alpha=0$.
\end{thm}

The statement that this is a self-adjoint elliptic boundary value problem has the following standard consequences. Define
\begin{equation} \tilde{H}= \{\alpha\in \Omega^{2}_{14}: \Delta\alpha=0, \alpha\Vert_{\partial, 8}=0, d^{*}\alpha\vert_{\partial M}=0\}. \end{equation}
Then
\begin{enumerate}
\item $\tilde{H}$ is finite dimensional;
\item a solution to the boundary value problem in Theorem 1 exists if and only if $\rho$ is $L^{2}$-orthogonal to $\tilde{H}$;
\item in such a case we have elliptic estimates
$$  \Vert \alpha\Vert_{L^{2}_{k}}\leq C_{k} \Vert \rho \Vert_{L^{2}_{k-2}}. $$
\end{enumerate}

The proof of Theorem 1 extends across the next few pages, including Lemmas 3 and 4. 

There is a standard definition of an elliptic boundary value problem (see \cite{kn:T}, Chapter 5, for example) but the general  theory is somewhat complicated and we do not need much of it here.
We take as known the theory of the Dirichlet problem for the Laplace operator on $\Omega^{2}_{14}$ and, for simplicity, we assume initially that the only solution $\alpha\in \Omega^{2}_{14}$ of $\Delta\alpha=0$ with $\alpha\Vert_{\partial M}=0$ is $\alpha=0$. Then, by the standard theory, for any $\rho\in \Omega^{2}_{14}$ and section $\theta$ of the restriction of $\Lambda^{2}_{14}$ to $\partial M$ there is a unique solution $\alpha=\alpha(\rho,\theta)$ of the equation $\Delta\alpha=\rho$ with $\alpha\Vert_{\partial M}= \theta$. Now consider a $1$-form $a$ on $\partial M$ and set $\theta(a)= 2 a\wedge \nu^{*}- \chi_{6}(a)$ as in Lemma 1. The boundary value problem in Theorem 1 then becomes an equation for $a$:
$$    d^{*}\alpha(\rho,\theta(a))\vert_{\partial M}=0. $$
Write $d^{*}\alpha(\rho,0)\vert_{\partial M}= -\sigma$ and let 
$$  P: \Omega^{1}(\partial M)\rightarrow \Omega^{1}(\partial M)$$ be the operator which maps $a$ to $d^{*}\alpha(\theta(a),0)\vert_{\partial M}$. Then  the equation to solve for $a$ is \begin{equation}P(a)= \sigma. \end{equation} The claim that the boundary value problem is elliptic is equivalent to the claim that $P$ is an elliptic pseudo-differential operator of order $1$ on $\partial M$. When this holds the equation $(22)$ can be solved for $a$ provided that $\sigma$ lies in a subspace of finite codimension and the solution $a$ is unique up to a finite dimensional kernel. These facts imply the corresponding statements about the boundary value problem. The simplifying assumption on the solubility of the Dirichlet problem  is unnecessary since the whole discussion can take place modulo finite dimensional subspaces. From another point of view we can run the arguments above in a  model  \lq\lq flat''  case (as in the  proof of Lemma 3 below) and use the solution there to construct a parametrix for our boundary value problem.

\

The ellipticity of the operator $P$ is a  condition  on the symbol and this symbol can be described as follows. 
\begin{lem}
Define an operator $\tilde{P}:\Omega^{1}(\partial M)\rightarrow \Omega^{1}(\partial M)$ by 
$$  \tilde{P}(a)=  2 \Delta^{1/2} a - d^{*}_{\partial M} \chi_{6}(a). $$
Then $P$ and $\tilde{P}$ have the same symbol.
\end{lem}

\begin{proof}
To see this we can consider a situation where the geometry is locally flat, so we can work in $\bC^{3}\times (-\infty, 0]$ with boundary $\bC^{3}=\bC^{3}\times \{0\}$ and co-ordinate $t$ in the $\bR$ factor. We write a $2$-form  as 
$$\alpha= 2 a_{t}\wedge dt - \Psi_{t}$$
where $a_{t}, \Psi_{t}$ are respectively $t$-dependent $1$-forms and $2$-forms on $\bC^{3}$. Then 
$$   d^{*}\alpha= 2 \frac{da_{t}}{dt} - d^{*}_{6}\Psi_{t} + 2 (d^{*}_{6}a_{t}) dt, $$ where $d^{*}_{6}$ denotes the $d^{*}$ operator on $\bC^{3}$. 
In our situation we have $\alpha\in \Omega^{2}_{14}$ so $\Psi_{t}= \chi_{6}a_{t} + \Theta_{t}$, where $\Theta_{t}$ takes values in $\Lambda^{2}_{\partial,8}$ and $\Theta_{0}=0$.   Thus  the restriction of $d^{*}\alpha$ to the boundary is given by
$$ 2\frac{da_{t}}{dt}\mid_{t=0} - d^{*}_{6} (\chi_{6}(a_{0})), $$
where $a(t)$ is the harmonic extension of the given $1$-form $a$ on the boundary. Thus $\frac{da}{dt}\mid_{t=0}$ is obtained from $a_{0}$ by applying the \lq\lq Dirichlet-to-Neumann'' operator,  whose symbol is the same as that of $\Delta^{1/2}$, and this gives the statement in the Lemma. \qed \end{proof}

The symbol of $\Delta^{1/2}$ at a  cotangent vector $\xi$ on $\partial M$ is multiplication by $\vert \xi\vert$. Thus the ellipticity of our boundary value problem follows from the following statement.
\begin{lem} If $\xi$ is a unit cotangent vector on $\partial M$ the symbol of $a\mapsto d^{*}_{\partial M} \chi_{6}(a)$ at $\xi$ has eigenvalues $0,\pm 1$ \end{lem}

\begin{proof} To prove this Lemma
we can compute in the flat model $\bC^{3}$ as in (4), with standard co-ordinates $z_{i}=x_{i}+ \sqrt{-1} y_{i}$,  with $\omega= dx_{1} dy_{1}+ dx_{2} dy_{2} + dx_{3}dy_{3}$ and with
$$ \rho= -dy_{1}dy_{2}dy_{3} +dy_{1}dx_{2}dx_{3}+ dx_{1}dy_{2}dx_{3}+ dx_{1}dx_{2}dy_{3}. $$
One finds from the definition that if $a= \sum_{i=1}^{3}\lambda_{i} dy_{i} + \mu_{i} dx_{i}$ then
$$ \chi_{6}(a)= \sum_{{\rm cyclic}} \mu_{i} (dy_{j} dy_{k}- dx_{j}dx_{k}) + \lambda_{i}(dy_{j} dx_{k}+ dx_{j}dy_{k}), $$
where the notation means that $(ijk)$ run over cyclic permutations of $(123)$. Since the symmetry group $SU(3)$ acts transitively on the unit sphere it suffices to check any given unit co-tangent vector $\xi$, so we take $\xi= dx_{1}$. In other words we have to pull out the $\partial_{1}=\frac{\partial}{\partial x_{1}}$ term in $d^{*}\chi_{6}(a)$. This is
$$   (\partial_{1} \mu_{3}) dx_{2} - (\partial_{1}\mu_{2}) dx_{3} +(\partial_{1}\lambda_{2}) dy_{3} - (\partial_{1} \lambda_{3}) dy_{2}. $$
To get the symbol we replace the derivative $\partial_{1}$ by multiplication by $\sqrt{-1}$. We see that the symbol at $dx_{1}$ is the linear map $\Sigma$ with

$$\begin{array}{ccc}  \Sigma(dx_{1})=0 & \Sigma(dx_{2})= -\sqrt{-1} dx_{3}& \Sigma(dx_{3})=\sqrt{-1} dx_{2}\\ \Sigma(dy_{1})=0& \Sigma(dy_{2})=\sqrt{-1}dy_{3}& \Sigma(dy_{3})=-\sqrt{-1} dy_{2}. \end{array}$$

This linear map $\Sigma$ has eigenvalues $0,1,-1$ (each with multiplicity two).\qed\end{proof}

\

\

Next we  establish the self-adjoint property. Let us denote the boundary conditions in the statement of Theorem 1 by $(BC)$. Recall that in general the adjoint boundary conditions $(BC)^{*}$ are defined by saying that $\beta\in \Omega^{2}_{14}$ satisfies $(BC)^{*}$ if and only if we have an equality of $L^{2}$ inner products
$\langle \Delta \alpha, \beta\rangle= \langle \alpha, \Delta \beta\rangle$ for all $\alpha$ satisfying $(BC)$. We want to show that the boundary conditions $(BC)^{*}$ are the same as $(BC)$. To do this it suffices, by a simple dimension-counting argument, to show that if $\alpha, \beta$ both satisfy $(BC)$ then  $\langle \Delta \alpha, \beta\rangle= \langle \alpha, \Delta \beta\rangle$. We can see this by an indirect argument using the operator $L$ given by (13).  First we claim that if
$\alpha, \beta\in \Omega^{2}_{14}$ satisfy $\alpha\Vert_{\partial, 8}=\beta\Vert_{\partial, 8}=0$ then $\langle L \alpha, \beta\rangle=\langle \alpha, L\beta\rangle$. Indeed choose $\alpha',\beta'\in \Omega^{2}_{7}$ such that $(\alpha+\alpha')\vert_{\partial M}, (\beta+\beta')\vert_{\partial M}$ vanish. Then
$$  \langle L\alpha, \beta\rangle= \langle L(\alpha +\alpha'), \beta+\beta'\rangle, $$
since $L$ vanishes on $\Omega^{2}_{7}$ and maps to $\Omega^{2}_{14}$. Let $\langle \ , \ \rangle_{Q}$ be the symmetric bilinear form associated to the quadratic form $Q$ on $\Omega^{2}_{M}$
\begin{equation}   \langle \gamma_{1}, \gamma_{2}\rangle_{Q}= \frac{4}{3} \langle d_{1} \gamma_{1}, d_{1}\gamma_{2}\rangle+ \langle d_{7}\gamma_{1}, d_{7}\gamma_{2}\rangle - \langle d_{27}\gamma_{1}, d_{27}\gamma_{2}\rangle. \end{equation}

Then {\it if $\gamma_{1}, \gamma_{2}$ vanish on $\partial M$} we have
\begin{equation} \langle \gamma_{1} \gamma_{2}\rangle_{Q}= \langle L\gamma_{1}, \gamma_{2}\rangle= \langle \gamma_{1} L\gamma_{2}\rangle. \end{equation}
We apply this to $\alpha+\alpha', \beta+\beta'$, which vanish on the boundary by construction, so  we have
$$  \langle L(\alpha+\alpha'), \beta+\beta'\rangle = \langle \alpha+\alpha', \beta+\beta'\rangle_{Q}$$ which is symmetric in $\alpha,\beta$. This we have shown that $\langle L\alpha, \beta\rangle$ is symmetric in $\alpha,\beta$. Now, using  Proposition 1 we can write
  $\Delta = L+ \sthreetwo d_{14} d^{*}$ on $\Omega^{2}_{14}$. If $\alpha, \beta$ satisfy the other part of $(BC)$ that is,  if  $d^{*}\alpha\vert_{\partial M}=d^{*}\beta\vert_{\partial M}=0$, then 
$$  \langle dd^{*}\alpha,\beta\rangle = \langle d^{*}\alpha, d^{*}\beta\rangle=\langle \alpha, dd^{*}\beta\rangle $$
since the relevant boundary terms vanish. This completes the proof of self-adjointness.

\

The last statement in Theorem 1 (that $d^{*}\rho=0$ implies $d^{*}\alpha=0$) is a particular case of the first item in Proposition 2. This completes the proof of Theorem 1. 


\

 The  operator $L$ and the symmetric form $\langle\ ,\ \rangle_{Q}$ are related by  a boundary term. For $\alpha\in \Omega^{2}_{14}$ with $\alpha\Vert_{\partial,8}=0$ we define a $1$-form on $\partial M$:
$$\alpha\Vert_{\partial,6}= \chi_{6}^{-1}(\alpha\vert_{\partial M}).$$ 
Recall also that we have a $2$-form $\omega$ on $\partial M$ given by the contraction of $\phi$ with the normal vector.

\begin{prop}
If $\alpha_{14}, \beta_{14}$ are in $ \Omega^{2}_{14}$ with $\alpha_{14}\Vert_{\partial, 8}, \beta_{14}\Vert_{\partial, 8}=0$ then 
$$  \langle L\alpha_{14},\beta_{14}\rangle= \langle \alpha_{14}, \beta_{14}\rangle_{Q} + \langle \alpha_{14}\Vert_{\partial, 6}, \beta_{14}\Vert_{\partial,6}\rangle_{\partial} $$
where, for $1$-forms $a,b$ on $\partial M$,
$$   \langle a, b\rangle_{\partial} = - \int_{\partial M} \chi_{6}(a)\wedge d(b \wedge \omega). $$

\end{prop}

It follows from this Proposition that, given $\rho\in \Omega^{2}_{14}$ with $d^{*}\rho=0$,  our linear boundary value problem is the Euler-Lagrange equation associated to the   functional on $A$ given by 
\begin{equation}  -\Vert d\alpha\Vert^{2}+ \langle \alpha\Vert_{\partial, 6}, \alpha\Vert_{\partial,6}\rangle_{\partial} - \langle \rho, \alpha \rangle. \end{equation}

\
\begin{proof}
We give a derivation of Proposition 5 although we will not use the result, so this is a digression from our main path. First, one can check that  $\langle\ ,\ \rangle_{\partial}$ is symmetric, so by polarisation it suffices to prove the formula for $\beta_{14}=\alpha_{14}$. As before, choose $\alpha_{7}$ so that $\alpha=\alpha_{14}+\alpha_{7}$ vanishes on the boundary.  It follows then by integration by parts that
$$   \langle \alpha,\alpha_{7}\rangle_{Q}= \langle \alpha, L\alpha_{7}\rangle =0$$
and
$$  \langle \alpha,\alpha_{14}\rangle_{Q}= \langle \alpha, L\alpha_{14}\rangle= \langle \alpha_{14}, L\alpha_{14}\rangle. $$
Thus
$$   \langle \alpha_{14},\alpha_{7}\rangle_{Q}=  - \langle \alpha_{7},\alpha_{7}\rangle_{Q} $$ and 
$$    \langle \alpha_{14},\alpha_{14}\rangle_{Q}= \langle \alpha_{14}, L\alpha_{14}\rangle +  \langle \alpha_{7},\alpha_{7}\rangle_{Q}. $$

Let $v$ be the vector field on $M$ such that $\alpha_{7}=i_{v}(\phi)$ and let $S:\Lambda^{3}\rightarrow \Lambda^{3}$ be the bundle map equal to $\sfourthree, +1, -1$ on the factors $\Lambda^{3}_{1}, \Lambda^{3}_{7}, \Lambda^{3}_{27}$ respectively. Thus by (10) the first variation of $\Theta(\phi)$ for a variation $\delta\phi$ in $\phi$ is $*S(\delta \phi)$. Take $\delta\phi=d\alpha_{7}= {\cal L}_{v}\phi$, so that by diffeomorphism invariance of the constructions $*S(d\alpha_{7})= {\cal L}_{v}(*\phi)= d i_{v}(*\phi)$. Now
$$  \langle \alpha_{7}, \alpha_{7}\rangle_{Q}= \langle d\alpha_{7}, S(d\alpha_{7})\rangle = \int_{M} d\alpha_{7}\wedge d(i_{v}(*\phi)), $$
and we can write this as a boundary integral
$$  \int_{\partial M} \alpha_{7}\wedge d (i_{v} *\phi). $$
On the boundary the assumption that $(\alpha_{7}+\alpha_{14})\vert_{\partial M}=0$ implies that $v$ is tangential to $\partial M$. Then 
$$ i_{v}(*\phi) = \shalf i_{v} (\omega^{2}) = -a\wedge \omega, $$  
where $a=-i_{v}(\omega)$. Thus, on the boundary $\alpha_{7}= -a \wedge \nu^{*} - \chi_{6}(a)$ and $\alpha_{14}\vert_{\partial M}= \chi_{6}(a)$, so $a=\alpha_{14}\Vert_{\partial,6}$ and the boundary term is the integral of $ - \chi_{6}(a) \wedge d(a\wedge \omega)$ as required.\qed\end{proof}

 \

{\bf Remarks}

\begin{itemize}
\item The ellipticity of our boundary value problem depends crucially on the factor $2$ appearing in the decomposition of $\Lambda^{2}_{14}$ in Lemma 1;  more precisely that this factor is not $\pm 1$. By contrast there is a similar-looking boundary value problem for forms in $\Omega^{2}_{7}$ which is {\it not} elliptic. For this we consider the equation $\Delta\gamma=0$ for $\gamma \in \Omega^{2}_{7}$ with boundary condition, in the obvious notation, $d^{*}\gamma\vert_{\partial M}=0$ and $\gamma\Vert_{\partial, 1}=0$. Under the identification between $\Omega^{2}_{7}$ and $\Omega^{1}$, the solutions correspond to $1$-forms $\eta$ on $M$ with $d_{7}\eta=0, d^{*}\eta=0$ and with $i_{\nu}(\eta)=0$ on $\partial M$. This is the gauge-fixed, abelian, \lq\lq $G_{2}$-instanton'' equation and the space of solutions is infinite-dimensional. 

\item By standard theory, there is a complete set of eigenfunctions associated to our problem {\it i.e.} solutions of $-\Delta \alpha= \lambda \alpha$ satisfying $(BC)$. The  spectrum is discrete and bounded above so there are only finitely many positive eigenvalues.  This 1-sided boundedness can be seen by considering the 1-parameter family of product metrics on $M\times S^{1}$ with the length of the $S^{1}$-factor equal to $\kappa$, lying in the interval $[2\pi,4\pi]$ say. We consider  sections of the bundle $\pi^{*}(\Lambda^{2}_{14})$ lifted by the projection $\pi:M\times S^{1}\rightarrow M$.  There is an obvious way to lift the boundary conditions $(BC)$ to $M\times S^{1}$ and we consider the operator 
$$  -\Delta_{M\times S^{1}} =-\Delta_{M} + \left(\frac{d}{d\theta}\right)^{2}, $$ with these lifted boundary conditions.  The same discussion as before shows that this is an elliptic boundary value problem. The crucial fact is that the eigenvalues of the symbol $\Sigma$ in the proof of Lemma 3 have modulus  {\it less than} $2$. If 
 the spectrum of our original problem is not bounded above there are  eigenfunctions $\alpha_{i}$ with
eigenvalues $\lambda_{i}\rightarrow \infty$. For all large $i$ we can choose parameter values $\kappa_{i}$ such that $\sqrt{\lambda_{i}}=0 {\rm mod}\ \kappa_{i}$
 Then the sections
$$ \tilde{\alpha}_{i}= \alpha_{i} \cos( \sqrt{\lambda_{i}}\ \theta)$$ satisfy $\Delta_{M\times S^{1}} \tilde{\alpha}_{i}=0$, and this plainly contradicts the elliptic estimate on $M\times S^{1}$, which holds uniformly for parameter values $\kappa\in [2\pi,4\pi]$. 

\end{itemize}

\section{The nonlinear problem}
\subsection{Gauge fixing}

So far in this paper the connection between the linear theory and the deformations of $G_{2}$-structures has only been made at a formal level. We now correct this and develop the full nonlinear theory. There are two aspects to the nonlinearity, the first involving the action of the diffeomorphism group and the second involving the nonlinear nature of the torsion-free condition. For the first, happily, we are able to refer to the careful treatment of Fine, Lotay and Singer in \cite{kn:FLS}. This treats a 4-dimensional problem, but the proofs, extending the well-known results of Ebin for closed manifolds,  go over immediately to our situation.

 The standard approach to constructing a slice for the action of the diffeomorphism group on some space of tensors is to consider, at a tensor $\tau$, the variations $\delta\tau$ which are $L^{2}$ orthogonal to the Lie derivatives $L_{v}\tau$, for all vector fields $\tau$. 
 Under the identification of tangent vectors with $\Lambda^{2}_{7}$ the Lie derivatives of the closed 3-form $\phi$ are the image of $d:\Omega^{2}_{7}\rightarrow \Omega^{3}$. So in our case the standard slice for the diffeomorphism group action is given by variations $\delta \phi$ with $\pi_{7}(d^{*}(\delta\phi))=0$, where $\pi_{7}$ is the projection to $\Lambda^{2}_{7}$. As usual, it is convenient to work with Banach spaces and following \cite{kn:FLS} we will use Sobolev spaces, although H\"older spaces would work just as well. We fix some suitably large $s$, say $s=5$,  and consider the set of maps from $M$ to $M$ which are equal to the identity on the boundary and with $s+1$ derivatives in $L^{2}$.
Such maps are $C^{1}$ and it makes sense to consider diffeomorphisms of this class. These diffeomorphisms form a topological group and we define ${\cal G}^{s+1}$ to be the identity component of these $L^{2}_{s+1}$ diffeomorphisms. The group ${\cal G}^{s+1}$ acts on the space of $L^{2}_{s}$ $3$-forms on $M$. With our choice of the Sobolev index $s$ these forms are also $C^{1}$. In particular the notion of positive $3$-form makes sense.
For $\delta>0$ let $$S_{\delta}= \{ \phi+ \chi: \Vert \chi\Vert_{L^{2}_{s}}<\delta, \pi_{7}d^{*}\chi=0\}. $$
\begin{prop} There are constants $\epsilon,\delta>0$ such that for every 3-form $\tilde{\phi}$ with $\Vert \tilde{\phi}-\phi\Vert_{L^{2}_{s}}<\epsilon
$ there is a unique diffeomorphism $f\in {\cal G}^{s+1}$  such that
$f^{*}(\tilde{\phi})$ lies in $S_{\delta}$. 
\end{prop}

The statement is modelled on Theorem 2.1 of \cite{kn:FLS} and the proof is essentially the same so we do not go into  it in detail here. However we do want to recall the linear result which underpins the proof.

\begin{lem}
For any $\sigma\in \Omega^{2}_{7}$ there is a unique $\gamma=\Gamma(\sigma)\in \Omega^{2}_{7}$
with $\gamma\Vert_{\partial M}=0$ and $\pi_{7}d^{*}d\gamma=\sigma$. The map $\Gamma$ extends to a bounded map from $L^{2}_{s-1}$ to $L^{2}_{s+1}$.
\end{lem}

\begin{proof}It is easy to check that $\pi_{7}d^{*}d$ is a self-adjoint elliptic operator
on $\Omega^{2}_{7}$. Thus by the standard theory it suffices to show uniqueness,
in other words that for $\gamma\in \Omega^{2}_{7}$ if $\pi_{7}d^{*}d\gamma=0$ and $\gamma\Vert_{\partial
M}=0$ then $\gamma=0$. Integrating by parts, the conditions imply that $d\gamma=0$.
Using the identification of $\Lambda^{2}_{7}$ with tangent vectors, $\gamma$
corresponds to a vector field $v$,  vanishing on the boundary of $M$, with
$L_{v}(\phi)=0$. This implies that  $v$ is  a Killing field for the metric
$g_{\phi}$ and it is a simple fact from Riemannian geometry that vanishing
on the boundary forces $v$ to vanish everywhere. \qed \end{proof}

The set of closed $3$-forms in a given enhancement class is preserved by the diffeomorphism group ${\cal G}$. Thus we immediately get from Proposition 5 a model for a neighbourhood in ${\cal Q}_{\hatrho}$. Define
$$  T=\{d\alpha: \alpha\in \Omega^{2}(M), \alpha\vert_{\partial M}=0, \pi_{7}d^{*}d\alpha=0\}, $$
and let $T_{s}$ be the $L^{2}_{s}$ completion of $T$. Let ${\cal P}^{s}_{\hatrho}$
be the $L^{2}_{s}$ version of ${\cal P}_{\hatrho}$, in an obvious sense,
and ${\cal Q}^{s}_{\hatrho}$ be the quotient by ${\cal G}^{s+1}$
Then Proposition 5 implies that for suitable small $\delta>0$ the map $\chi\mapsto \phi+\chi$ induces a homeomorphism from the ball $\{ \chi\in T_{s}: \Vert \chi\Vert_{L^{2}_{s}}<\delta\}$ to a neighbourhood of $[\phi]\in {\cal Q}^{s}_{\hatrho}$. Slightly more generally, if a 3-form $\phi_{1}$ is sufficiently close to $\phi$ in $L^{2}_{s}$ norm and if $\hatrho_{1}$ is the corresponding enhancement class,  the map $\chi\mapsto \phi_{1}+\chi$ induces a homeomorphism from this same ball in $T_{s}$ to a neighbourhood of $[\phi_{1}]$ in ${\cal Q}^{s}_{\hatrho_{1}}$.

It follows immediately from Lemma 5 that the natural map from $T$ to $T{\cal Q}$ is an isomorphism. Define a vector space
$$  B=\{ \alpha=\alpha_{7}+\alpha_{14}\in \Omega^{2}(M): \alpha_{14}\in A, \alpha \vert_{\partial M}=0, \pi_{7}d^{*}d\alpha=0\}. $$
The map $\alpha\mapsto\alpha_{14}$ induces a map $p:B\rightarrow A$ and Lemma 5 implies that this is an isomorphism. To spell this out, it is clear from the decomposition of the forms on the boundary that we can choose a smooth bundle map $\tau: \Lambda^{2}_{14}\rightarrow \Lambda^{2}_{7}$, supported in a neighbourhood of the boundary, such that for all $\beta\in \Omega^{2}_{14}$ we have $\beta+\tau(\beta)\vert_{\partial M}=0$. Now for $\alpha_{14}\in A$ define
\begin{equation}  F(\alpha_{14})= \alpha_{14}+ \tau(\alpha_{14}) - \Gamma( \pi_{7}d^{*}d(\alpha_{14}+\tau(\alpha_{14})).\end{equation}
Then $F$ maps to $B$ and is the inverse to $p$. It is clear from the formula that this extends to an isomorphism (of topological vector spaces) $F:A_{s+1}\rightarrow B_{s+1}$ where $B_{s+1}$ is the $L^{2}_{s+1}$ completion of $B$. Now the exterior derivative induces a  map from $B$ to $T$ and it follows from Proposition 3 and the above-noted isomorphism of $T$ and $T{\cal Q}$ that this is an isomorphism and also for the corresponding Sobolev versions. Putting all this together we have the following result.

 \begin{prop}
 There is a $\delta'>0$ such that, if $\phi_{1}$ is sufficiently close to $\phi$ in $L^{2}_{s}$,  the map $\alpha\mapsto \phi_{1}+ d F(\alpha)$
induces a homeomorphism from the ball in $A_{s+1}$: $\{\alpha\in A_{s+1}: \Vert \alpha\Vert_{L^{2}_{s+1}}<
\delta'\}$ to a neighbourhood of $[\phi_{1}]$ in ${\cal Q}^{s}_{\hatrho_{1}}$ where $\hatrho_{1}$ is the enhanced boundary value determined by $\phi_{1}$.
\end{prop}

\subsection{A Fredholm equation}

We want to represent the solutions of our boundary value problem, for given data $\hatrho$, as the zeros of a Fredholm map. Recall that for $\alpha\in\Omega^{2}(M)$ we defined $W(\alpha)= *_{\phi} d \Theta(\phi+d\alpha)$. In our description from Proposition 7 of a neighbourhood in ${\cal Q}^{s}_{\hatrho}$ the torsion-free equation is $W(F(\alpha))=0$, for small $\alpha\in A_{s+1}$. 
 The complication is that  for general $\alpha$ the term  $W(F(\alpha))$ lies in the subspace $*_{\phi}\Omega^{5}_{14,\tilde{\phi}}$ where $\tilde{\phi}= \phi+d F(\alpha)$ and this space also depends on $\tilde{\phi}$. We use a projection construction to get around this,  which essentially amounts to constructing a local trivialisation of the cotangent bundle of ${\cal Q}^{s}_{\hatrho}$.

For any $\alpha$ as above, write $\sigma= W(F(\alpha))=*_{\phi} d\tilde{\phi}$. Then $d^{*}\sigma=0$ and $\sigma$ is a section of the bundle $*_{\phi} \Lambda^{5}_{14,\tilde{\phi}}$. Let $\sigma=\sigma_{7}+\sigma_{14}$. By our Hodge theory result, Proposition 2,  there is a unique $\eta\in\Omega^{1}$ orthogonal to ${\cal H}^{1}$  solving the equation $ d^{*}d\eta= d^{*}\sigma_{7}$ with  $\eta\vert_{\partial M}=0$.  Now recall (as in the proof of Proposition 3)  that $d^{*}d= \sthreetwo d^{*} d_{14}$ on $\Omega^{1}$. Thus
$$  d^{*}\sigma_{14}=-d^{*}\sigma_{7} = - \sthreetwo d^{*} d_{14}\eta  .$$

Set $\tilde{\sigma}_{14}= \sigma_{14}+ \sthreetwo d_{14}\eta$, so $\tilde{\sigma}_{14}$ lies in $\Omega^{2}_{14}$ and satisfies $d^{*}\tilde{\sigma}_{14}=0$.

\begin{lem} There is a $\delta''>0$ such that if $\Vert \alpha\Vert_{L^{2}_{s+1}}< \delta''$ then $\sigma=0$ if and only if $\tilde{\sigma}_{14}=0$. 
\end{lem}

\begin{proof}We know that $\sigma$ is a section of the bundle $*_{\phi}\Lambda^{5}_{14,\tilde{\phi}}$. When $\tilde{\phi}=\phi$ this is exactly the bundle $\Lambda^{2}_{14}$. If $\tilde{\phi}$ is close to $\phi$ in $C^{0}$ we can use the standard graph construction. There is a bundle map $$S_{\tilde{\phi}}:\Lambda^{2}_{14}\rightarrow \Lambda^{2}_{7}$$ 
such that elements of  $*_{\phi}\Lambda^{5}_{14,\tilde{\phi}}$ are of the form $\tau_{14}+ S_{\tilde{\phi}} \tau_{14}$. If $\alpha$ is small in $L^{2}_{s+1}$ then $S_{\tilde{\phi}}$ is small in $L^{2}_{s}$. In the preceding discussion, we have $\sigma_{7}= S_{\tilde{\phi}} \sigma_{14}$ so
$$  \Vert \sigma_{7} \Vert_{L^{2}_{k}} \leq \epsilon \Vert \sigma_{14}\Vert_{L^{2}_{k}}, $$
for $k\leq s-1$, 
where we can make $\epsilon$ as small as we please by taking $\alpha$ small in $L^{2}_{s+1}$. On the other hand, the elliptic estimates for the boundary value problem give 
$$   \Vert d_{14} \eta \Vert_{L^{2}_{k}}\leq C \Vert \sigma_{7} \Vert_{L^{2}_{k}}. $$
 The Lemma follows by choosing $\delta''$ such that  $\sthreetwo C\epsilon<1$.\qed\end{proof}
 
 Write $\tilde{\sigma}_{14}= \Pi_{\tilde{\phi}}(\sigma)$. The conclusion of Proposition 7 and Lemma 6 is that the solutions of the torsion free equation in a neighbourhood of $[\phi]$ in ${\cal Q}_{\hatrho}$ correspond to the zeros of a map
${\cal F}$ defined by 
\begin{equation}  {\cal F}(\alpha)= \Pi_{\tilde{\phi}}\left( W(\phi+ d F \alpha)\right). \end{equation}
 This map ${\cal F}$ takes values in the space $A'=\{\alpha\in \Omega^{2}_{14}: d^{*}\alpha=0\}$. Write $A'_{s-1}$ for the $L^{2}_{s-1}$ completion.

\begin{prop} ${\cal F}$ extends to a smooth Fredholm map of index $0$ from a neighbourhood of the origin in $A_{s+1}$ to $A'_{s-1}$. The derivative at $\alpha=0$ is $-\Delta: A_{s+1}\rightarrow A'_{s-1}$. 
\end{prop}\begin{proof}
The proof is straightforward, given Theorem 1. We compute the derivative formally:
$$   W(F(\alpha))= L(F(\alpha))+ O(\alpha^{2}), $$
(by definition of the linearised operator $L$);
$$   L(F(\alpha))= L(\alpha),$$
(since $F(\alpha)$ differs from $\alpha$ by a term in $\Omega^{2}_{7}$ and $L$ vanishes on $\Omega^{2}_{7}$);
$$  L(\alpha)=-\Delta\alpha, $$
(by the discussion in Section 3). Clearly the derivative of the projection term $\Pi_{\tilde{\phi}}$ at $\alpha=0$ is the identity on $\Omega^{2}_{14}$. Now one can check that these calculations are compatible with the Sobolev structures.

 We know that $\Delta$ is self-adjoint on $A$ so the index is zero and the cokernel can be identified with the kernel. This kernel is $${\cal H}_{\phi}=A\cap \tilde{H}= \{\alpha\in \Omega^{2}_{14}: d^{*}\alpha=0, \Delta\alpha=0, \alpha\Vert_{\partial,8}=0\}, $$ and the isomorphism from $A$ to $T{\cal Q}$ takes ${\cal H}_{\phi}$ to the space $H_{\phi}$  defined in (17). \qed\end{proof}

This proposition achieves our goal of representing the solutions of the torsion-free equation in ${\cal Q}^{s}_{\hatrho}$ as the zeros of a Fredholm map. By standard elliptic regularity any $L^{2}_{s+1}$ solution is smooth.

\

 To complete the discussion, we consider varying the boundary data $\hatrho$. Let $\theta$ be a closed $3$-form on $M$ which is small in $L^{2}_{s}$, so that $\phi+\theta$ is a closed positive $3$-form  on $M$ which defines  perturbed enhanced boundary data $\hatrho(\theta)$. By Proposition 7, the map
$$  \alpha\mapsto \phi+\theta+ dF(\alpha)$$
gives a homeomorphism from the $\delta''$-ball in $A^{s+1}$ to a neighbourhood of $[\phi+\theta]$ in ${\cal Q}_{\hatrho(\theta)}$. In this neighbourhood the solutions of the torsion-free equation correspond to the zeros of a perturbed map 
$$ {\cal F}_{\theta}(\alpha) =  \Pi_{\tilde{\phi}}\left( W(\phi+ \theta+ d F \alpha)\right), $$
where $\tilde{\phi}=\phi+\theta+ dF(\alpha)$. 
This is a smooth map in the two variables $\theta\in {\rm ker}\ d \cap L^{2}_{s}, \alpha\in A_{s+1}$,  taking values in $A'_{s-1}$. Thus we can apply standard theory to study the solutions of the torsion-free equation for nearby boundary data. In particular we have by the implicit function theorem:
\begin{thm} 
If the vector space $H_{\phi}$ is zero then for smooth $\theta$ which are sufficiently small in $L^{2}_{s}$ there is a unique solution of the torsion-free equation in ${\cal Q}_{\hatrho(\theta)}$ close to $\rho+\theta$. 
 \end{thm}

 More generally, if $H_{\phi}$ is not zero the standard theory of Fredholm maps gives a finite-dimensional \lq\lq Kuranishi model'' for the solutions of the torsion-free equation.

This discussion of the local structure in ${\cal Q}_{\hatrho}$ can be extended to the Hitchin functional. We define a quadratic form on $A$ by
  \begin{equation}  q_{\phi}(\alpha)= -\langle \alpha, \Delta \alpha\rangle. \end{equation}
So the eigenvalues associated to  our boundary value, which we discussed in Section 3, are the eigenvalues of the quadratic form $q_{\phi}$ relative to the $L^{2}$ form.  
Standard theory (as described in \cite{kn:D-1}, Proposition 2.5,  for example) gives a diffeomorphism from one neighbourhood of the origin in $A_{s+1}$ to another which takes the Hitchin functional to a sum $f\circ \pi  + \stthird q_{\phi}$ where $f$ is a real-valued function on ${\cal H}_{\phi}$ and $\pi: A_{s+1}\rightarrow {\cal H}_{\phi}$ is $L^{2}$ projection. 

We emphasise that the crucial difference in this case of manifolds with boundary, compared with the closed case,  is that the quadratic form $q_{\phi}$ is not manifestly negative definite due to the boundary term in Proposition 5.

\section{Examples and questions}

We have seen that the deformation theory for our boundary value problem is governed by a finite-dimensional vector space
$$  H_{\phi}= \frac{\{ d\alpha: \alpha\vert_{\partial M}=0, L(\alpha)=0\}}{ \{d\alpha_{7}: \alpha_{7}\vert_{\partial M}=0\}}. $$
Now we define another vector space
\begin{equation} K_{\phi}=\frac{ \{ d\alpha_{7}: \alpha_{7}\in \Omega^{2}_{7}(M)\}\cap\{ d\alpha : \alpha\in \Omega^{2}(M), \alpha\vert_{\partial M}=0\}}{ \{d\alpha_{7}: \alpha_{7}\vert_{\partial M}=0\} }. \end{equation}
There is an obvious map $E: K_{\phi}\rightarrow H_{\phi}$ arising from the fact that $L(\alpha_{7})$ vanishes for any $\alpha_{7}\in \Omega^{2}_{7}$ and it also obvious that this map $E$ is injective.  

 Next recall that we have a decomposition of forms on the boundary 
 $$\Omega^{2}(\partial M)= \Omega^{2,\partial}_{ 1}\oplus\Omega^{2,\partial}_{6}\oplus \Omega^{2,\partial}_{ 8},$$
and write $$\Omega^{2,\partial}_{7}= \Omega^{2,\partial}_{1}\oplus \Omega^{2,\partial}_{ 6}.$$
Thus the forms in $\Omega^{2,\partial}_{7}$ are exactly the restrictions to the boundary of forms in $\Omega^{2}_{7}(M)$. We define
$$  V_{\phi}= \{\theta\in \Omega^{2,\partial}_{7}: d\theta=0\}, $$
and $$W_{\phi}= \{ \alpha_{7}\in \Omega^{2}_{7}(M): d\alpha_{7}=0\}. $$ So there is a restriction map
$$ \iota: W_{\phi}\rightarrow V_{\phi}. $$
The space $W_{\phi}$ corresponds to the vector fields on $M$ preserving $\phi$ (which are Killing fields for the metric) and, as we have noted in the proof of Lemma 6, the map $\iota$ is an injection.

The exact sequence of the pair $(M,\partial M)$ gives a co-boundary map from $H^{2}(\partial M)$ to $H^{3}(M,\partial M)$. Since an element of $V_{\phi}$ defines a class in $H^{2}(\partial M)$ we have a map, which we denote by
$$   p: V_{\phi}\rightarrow H^{3}(M,\partial M).$$
\begin{prop}
There is an isomorphism
$$ K_{\phi}\cong {\rm ker}\ p/({\rm ker}\ p\cap {\rm Im}\  \iota). $$
\end{prop}
\begin{proof}
For simplicity, we just prove that if the right hand side is zero then so also is $K_{\phi}$ (which is what we will use). So suppose that we have a pair $\alpha, \alpha_{7}$ representing a class in $K_{\phi}$---i.e $d\alpha=d\alpha_{7}$ and $\alpha\vert_{\partial =0}$. Thus the restriction of $\alpha_{7}$ to the boundary is a $2$-form, $\theta$ say,  in $V_{\phi}$. Recall that in general the definition of the boundary map is that we extend $\theta$ to some 2-form $\Theta$ over $M$ and take the cohomology class of $d\Theta$ in $H^{3}(M,\partial M)$. In our case we can take $\Theta=\alpha_{7}$ and the fact that $d\alpha_{7}=d\alpha$ with $\alpha\vert_{M}=0$  says exactly  that $p(\theta)=0$. So by assumption
$\theta$ lies in the image of $\iota$, say $\theta=\iota(\tilde{\alpha}_{7})$. But now we can replace $\alpha_{7}$ by $\alpha_{7}-\tilde{\alpha}_{7}$, representing the same class in $K_{\phi}$. Thus we may as well suppose that $\alpha_{7}$ restricts to zero on the boundary, which means that the class in $K_{\phi}$ is zero.\qed\end{proof}

\begin{thm}

Suppose that $M$ is a domain with smooth boundary in a closed manifold with torsion-free $G_{2}$ structure $(M^{+},\phi^{+})$ and that $\phi$ is the restriction of $\phi^{+}$.
Then $E:K_{\phi}\rightarrow H_{\phi}$ is an isomorphism.
\end{thm}

 As we mentioned before, injectivity is trivial so we have to prove that the map is surjective. The proof is simply to extend deformations from $M$ to $M^{+}$ and then apply the standard theory on the closed manifold. The only complication is that we have to work with forms that are not smooth.  

Let $\alpha\in \Omega^{2}(M)$ represent a class in $H_{\phi}$, so $\alpha\vert_{\partial M}=0$ and $L(\alpha)=0$. The first step is to show that we can suppose that$\alpha$ satisfies the stronger condition $\alpha\Vert_{\partial M}=0$. Indeed suppose that in a normal product neighbourhood, with normal coordinate $t$,
$$   \alpha =  a_{t} dt + b_{t}, $$
where $a_{t}, b_{t}$ are $t$-dependent forms on $\partial M$. The hypothesis is that $b_{0}=0$. Let $\epsilon$ be the $1$-form $t a_{0} $ in this product neighbourhood, extended smoothly over $M$. Then $d\epsilon= t da_{0} t - a_{0} dt$ and $\alpha'=\alpha+ d\epsilon$ satisfies $\alpha'\Vert_{\partial M}=0$. Since $\alpha'$ represents the same class in $H_{\phi}$ we may as well suppose that $\alpha\Vert_{\partial M}=0$.

Next let $\ualpha$ be the 2-form on $M^{+}$ equal to $\alpha$ on $M$ and extended by zero over the complement. This extension is not smooth but it is Lipschitz, so $\ualpha\in L^{p}_{1}(M^{+})$ for all $p$. We apply the standard theory, as sketched in Section 2, to $\ualpha$. Thus we solve the equation $\Delta \eta= d^{*}\ualpha_{14}$.  The Lipschitz condition means that $d^{*}\ualpha_{14}$ has no distributional component and  standard elliptic theory gives $\eta\in L^{p}_{2}(M^{+})$ for all $p$.  We find a harmonic form $h_{14}$ on $M^{+}$ such that $\tilde{\alpha}_{14}= \ualpha_{14} -d_{14} \eta +h_{14}$ is orthogonal to the harmonic space. Then by construction $d^{*}\tilde{\alpha}_{14}=0$.
Now we bring in  the hypothesis that $L(\alpha)=0$. Taking the inner product with $\alpha$ this implies that $Q(\ualpha)=0$ (since the relevant boundary term vanishes). The arguments for smooth forms all extend to $L^{p}_{1}$ forms,  and we deduce that $Q(\tilde{\alpha}_{14})= 0$. But since $d^{*}\tilde{\alpha}_{14}=0$ we have $Q(\tilde{\alpha}_{14})= \Vert d\tilde{\alpha}_{14}\Vert^{2}$ so $d\tilde{\alpha_{14}}=0$ and hence $\tilde{\alpha}_{14}$ is harmonic. But this means that $\tilde{\alpha}_{14}$ vanishes, since it was chosen to be orthogonal to the harmonic forms.

Then $$d\ualpha= d(\ualpha_{7}+ \ualpha_{14}+ d_{7}\eta +d_{14}\eta + h)=d(\tilde{\alpha}_{7})$$
with $\tilde{\alpha}_{7}= \ualpha_{7}+ d_{7}\eta$.  If we know that $\tilde{\alpha}_{7}$ is smooth on the manifold-with-boundary  $M$, then we have shown that the pair $\alpha,\tilde{\alpha}_{7}$ represents a class in $K_{\phi}$ mapping by $E$ to the given class in $H_{\phi}$,  thus completing the proof of the theorem.  This smoothness follows from the following Lemma.

\begin{lem} Suppose $\tilde{\alpha}_{7}$ is an $L^{p}_{1}$ section of $\Lambda^{2}_{7}$ over $M^{+}$ such that $d\tilde{\alpha}_{7}$ is smooth up to the boundary on $M$. Then $\tilde{\alpha}_{7}$ is also smooth up to the boundary on $M$.
\end{lem}
 \begin{proof}Let $D$ be the operator $d:\Omega^{2}_{\partial, 7}\rightarrow \Omega^{3}(\partial M)$. This is an overdetermined elliptic operator, so if $\theta$ lies in some Sobolev space on $\partial M$ and $D\theta$ is smooth on $\partial M$ then $\theta$ is also smooth. We apply this to the restriction of $\tilde{\alpha}_{7}$ to $\partial M$. This lies  {\it a priori} in a Sobolev space $L^{p}_{1-1/p}$ but now we see that it is a smooth form on $\partial M$. The relation between $d^{*}$ and $d_{7}$ on $\Omega^{2}_{7}$ shows that $d^{*}\tilde{\alpha}_{7}$ is also smooth up to the boundary on $M$. So the same is true of $\Delta\tilde{\alpha}_{7}$. Thus $\tilde{\alpha}_{7}$ solves an elliptic boundary value problem
$$  \Delta\tilde{\alpha}_{7}= \rho, \tilde{\alpha}_{7}\vert_{\partial M}=\sigma, d^{*}\tilde{\alpha}_{7}\vert_{\partial M} =\tau, $$
with $\rho$ smooth up to the boundary on $M$ and $\sigma, \tau$ smooth on $\partial M$. It follows by elliptic regularity that $\tilde{\alpha}_{7}$ is smooth up to the boundary on $M$.  \qed\end{proof}

The example discussed in \cite{kn:D} is an annular region in ${\bf R}^{7}$ which can be embedded in a compact torus,  so Theorem 3 applies. In that example $K_{\phi}$ is not zero,  so the same is true of $H_{\phi}$ and the deformation problem is obstructed.

In a similar vein we have

\

\begin{prop} For $(M,\phi)\subset (M^{+},\phi^{+})$ as in Theorem 3 the quadratic form $q_{\phi}$ is negative semi-definite.
\end{prop}

This is essentially the same (in slightly different language) as \cite{kn:D}, Proposition 1. 

\

These results raise the following questions.

\

{\bf Question 1}
{\it Is it true that for all $(M,\phi)$  we have $H_{\phi}=K_{\phi}$ ?}

\

{\bf Question 2} {\it Is it true that for all $(M,\phi)$  the form $q_{\phi}$ is negative semi-definite?} 

\

The author has spent some effort attempting to answer these questions,  without success. By the same argument as in Proposition 1 of \cite{kn:D} an equivalent form of Question 2 is to ask whether the inequality 
$$   \Vert d_{27}\alpha_{14}\Vert\geq \Vert d_{7}\alpha_{14}\Vert $$
holds for all {\it compactly supported} $\alpha_{14}\in \Omega^{2}_{14}$.

The equation $D\theta=0$ (in the notation of Lemma 7) is highly overdetermined, so one expects that typically the space $V_{\phi}$ is $0$, and hence also $K_{\phi}$. If the answer to Question 1 above is affirmative it would follow that in most situations the space $H_{\phi}$ is zero i.e. that the same three facts for the closed manifold theory reviewed in the introduction hold for the boundary value problem, in most situations.

Leaving aside this question aside: Theorem 3  and Proposition 9 can be used to supply examples where $H_{\phi}$ vanishes. We will just consider one class of examples here. Let $N$ be a closed Calabi-Yau $3$-fold i.e. a complex 3-fold with K\"ahler form $\omega$ and holomorphic $3$-form $\Theta$ such that $\omega$ and $\rho={\rm Im}(\Theta)$ are equivalent to the standard model (4) at each point. For $L>0$ let $M_{L}$ be the manifold with boundary
$N\times [0,L]$ with $3$-form $\phi= \omega dt + \rho$. 

\begin{lem}
For this $\phi$ on $M_{L}$ we have $H_{\phi}=0$.
\end{lem}
\begin{proof}
We can embed $M_{L}$ in a closed manifold $M\times S^{1}$ so Theorem 3 applies. Thus we have to identify the space $V_{\phi}$ which is the sum of two copies of the kernel ${\rm ker}\ D$ of the operator $D$ on $N$. Taking cohomology we have a map $h:{\rm ker}\  D \rightarrow H^{2}(N)$ and it follows from the Hodge decomposition that the image of this is $V= \bR [\omega] + H^{2,0}_{\bR}$ where $H^{2,0}_{\bR}$ denotes the real part of complex cohomology. We claim that $h$ is  injective,  so that ${\rm ker}\ D$ is isomorphic to $V$. Suppose that $\theta$ lies in the kernel of $h$, so $\theta= d\eta$ for some $1$-form $\eta$ on $N$. In other words the component of $d\eta$ in $\Lambda^{1,1}_{0}$ vanishes. The Hodge-Riemann bilinear relations give that
$$ d\eta \wedge d\eta \wedge  \omega= \left(2 \vert d_{1} \eta\vert^{2} +  \vert d_{6}\eta\vert^{2}\right) {\rm vol}_{6},$$
(where $d_{1}$ denotes the component in $\bR\omega$ and $d_{6}$ the component in $\Lambda^{2,0}_{\bR}$). So we have, by Stokes theorem,
$$  0=\int_{N} d\eta\wedge d\eta\wedge \omega= 2 \Vert d_{1} \eta\Vert^{2}+  \Vert d_{6}\eta\Vert^{2}. $$
Hence $d\eta=0$ and the claim is proved.

We now have $V_{\phi}=V \oplus V$ with one copy of $V$ for each boundary component. In this case $H^{3}(M_{L}\partial M_{L})= H^{2}(N)$ and the map
$p:V_{\phi}\rightarrow H^{3}(M_{L}\partial M_{L})$ is 
$$p(\theta_{1}, \theta_{2})= \theta_{1}-\theta_{2}.$$
So the kernel of $p$ is the diagonal copy of $V$ in $V\oplus V$. On the other hand it is clear that the space $W_{\phi}$ is isomorphic to $V$ and that $\iota$ maps on to the diagonal so we see from Proposition 9 that $K_{\phi}=0$.\qed\end{proof}

We can apply our main result to get an  existence theorem for deformations of these product manifolds. To state this we need to pin down the choice of enhancement data. Recall that the space of enhancements is an affine space modelled on $H^{3}(M,\partial M)$ but with no canonical origin. In our case we have $H^{3}(M_{L},\partial M_{L})= H^{2}(N)$, as above. Fix   $2$-cycles $\sigma_{a}$ in $N$ representing a basis for $H_{2}(N)$. Then for any $3$-form $\psi$ on $M_{L}$ we define
$$  I_{a}=\int_{\sigma_{a}\times [0,L]} \psi. $$
The collection of these integrals can be regarded as an element $I(\psi)\in H^{2}(N)$. This induces an identification between the enhancements of a given boundary  form and $H^{2}(N)$. Clearly $I(\phi_{L})=L[\omega]$. 
\begin{thm}
For $L$ and $(N,\omega,\rho)$ as above there is a neighbourhood $U$ of  $\rho$ in the space of closed forms on $N$ (in the $C^{\infty}$ topology) and a neighbourhood $U'$ of $L[\omega]$ in $H^{2}(N)$  such that if $\rho_{0}, \rho_{L}$ are in $U$ and define the same cohomology class in $H^{3}(N)$ and $\nu$ is in $U'$ then there is a torsion free $G_{2}$-structure $\phi$ on $M_{L}$ which restricts to $\rho_{0}, \rho_{L}$ on the two boundary components  and with $I(\phi)=\nu$.
\end{thm}

Of course there is also a uniqueness statement, for solutions close to $\phi_{L}$.

\end{document}